\documentclass[11pt]{article}
\usepackage[T1]{fontenc}
\usepackage{lmodern}
\usepackage{microtype}
\usepackage{amsmath,amssymb,amsthm}
\usepackage{needspace}
\usepackage[a4paper,margin=1in]{geometry}
\usepackage{xcolor}
\usepackage{hyperref}

\hypersetup{
  colorlinks=true,
  linkcolor=blue!50!black,
  citecolor=blue!50!black,
  urlcolor=blue!50!black,
  pdftitle={A nowhere-zero point for several linear mappings simultaneously},
  pdfauthor={Amir Jafari},
  pdfkeywords={nowhere-zero vector, simultaneous linear mappings, Frobenius powers, forbidden values, matroid characteristic polynomial, blocking set, Alon-Jaeger-Tarsi conjecture}
}

\newtheorem{theorem}{Theorem}[section]
\newtheorem{lemma}[theorem]{Lemma}
\newtheorem{proposition}[theorem]{Proposition}
\newtheorem{corollary}[theorem]{Corollary}
\newtheorem{conjecture}[theorem]{Conjecture}
\newtheorem{question}[theorem]{Question}
\newcommand{\F}{\mathbb F}
\newcommand{\rk}{\operatorname{rk}}
\newcommand{\si}{\operatorname{si}}
\newcommand{\PG}{\operatorname{PG}}
\newcommand{\rmax}{r_{\max}}

\title{A nowhere-zero point for several linear mappings simultaneously}
\author{Amir Jafari\\[3pt]
  \small New Uzbekistan University\\[-1pt]
  \small\href{mailto:a.jafari@newuu.uz}{\texttt{a.jafari@newuu.uz}}}
\date{}

\begin{document}
\maketitle

\begin{abstract}
Let $q=p^k$, and let $A_1,\ldots,A_{r-1}$ be invertible $n\times n$
matrices over $\F_q$.  We prove that, if $k\ge r$, there is a vector $x$ for
which
\[
 x,A_1x,\ldots,A_{r-1}x
\]
are all nowhere zero.  For $r=2$ this recovers the theorem of Alon and Tarsi
over nonprime finite fields.  The proof tracks one monomial in the product
of the coordinate forms.  Frobenius powers keep every exponent below $p^r$,
and finite-field polynomial nonvanishing supplies the required vector.  The
same method treats rectangular matrices with independent rows and prescribed
forbidden values.

The method also gives an extension-degree criterion for representable
matroids and clarifies an unpublished prime-field conjecture of
M.~J.~Moghaddamzadeh.  Projective-geometric examples explain why the
analogous field-size statement fails over proper extensions and, translated
back to matrices, give lower bounds for the large-field problem.
\end{abstract}



\section{Introduction}

Call a vector \emph{nowhere zero} if none of its coordinates vanishes.  Alon
and Tarsi asked whether, for an invertible matrix $A$ over a finite field
$\F_q$, there must be a vector $x$ such that both $x$ and $Ax$ are nowhere
zero.  They conjectured that the answer is yes for every $q\ge4$ and proved
it whenever $q$ is not prime \cite{AT1989}.

We ask for one vector that is nowhere zero and remains nowhere zero under
several invertible mappings.  Given
$A_1,\ldots,A_{r-1}\in\operatorname{GL}_n(\F_q)$, we seek an
$x\in\F_q^n$ such that $x,A_1x,\ldots,A_{r-1}x$ are all nowhere zero.  Our
first theorem gives a condition depending only on the extension degree of
the field.

\begin{theorem}[Simultaneous nowhere-zero theorem]
\label{thm:simultaneous}
Let $q=p^k$, where $p$ is prime, let $r\ge2$ and $n\ge1$, and let
$A_1,\ldots,A_{r-1}\in\operatorname{GL}_n(\F_q)$.  If $k\ge r$, then
there is an $x\in\F_q^n$ such that
\[
 x,A_1x,\ldots,A_{r-1}x
\]
are all nowhere zero.
\end{theorem}

Here the identity is counted among the $r$ mappings.  Thus $r=2$ is exactly
Alon and Tarsi's theorem over nonprime finite fields.  Equivalently, for any
$B_1,\ldots,B_r\in\operatorname{GL}_n(\F_q)$ there is a $z$ for which every
$B_iz$ is nowhere zero.  Indeed, apply the theorem to
$B_2B_1^{-1},\ldots,B_rB_1^{-1}$ and put $z=B_1^{-1}x$.

The theorem is encoded by a simple polynomial.  Put $A_0=I_n$ and write
$\ell_{ij}(X)=(A_iX)_j$.  Then
\begin{equation}\label{eq:intro-polynomial}
 P(X)=\prod_{i=0}^{r-1}\prod_{j=1}^n\ell_{ij}(X)
\end{equation}
satisfies $P(x)\ne0$ precisely when all the displayed vectors are nowhere
zero.  For each $i$, the forms
$\ell_{i1},\ldots,\ell_{in}$ form a basis of $(\F_q^n)^*$.  The proof
therefore reduces to one algebraic fact: in characteristic $p$, the product
of $r$ bases of linear forms contains a nonzero monomial in which every
exponent is smaller than $p^r$.  When $r\le k$, these exponents are smaller
than $q=p^k$, and the coefficient form of the Combinatorial Nullstellensatz
gives the required point.

The extension-degree result and the known large-field results concern
different regimes.  Nagy and Pach proved the two-mapping statement over every
finite field of size greater than $61$, except possibly $79$
\cite{NagyPach}.  More generally, Nagy, Pach, and Tomon proved that, for each
fixed $r$, the simultaneous conclusion holds over every sufficiently large
finite field \cite[Theorem~1.7]{NagyPachTomon}.  Their theorem imposes no
condition on the extension degree.  Theorem~\ref{thm:simultaneous} instead
applies to every $\F_{p^k}$ with $k\ge r$, independently of $p$ and $n$.

The same argument applies to rectangular matrices with independent rows and
allows a prescribed set of forbidden values in every output coordinate.  We
prove this after the bounded-monomial estimate.

The proof uses only the independence relations among the coordinate forms.
Matroid language isolates this information.  It leads to the unpublished
prime-field conjecture of M.~J.~Moghaddamzadeh, exact obstruction thresholds in low rank, and
higher-rank bounds obtained by projection.  Embedded projective
subgeometries explain why the corresponding field-size statement fails over
proper extensions and, at the end, return us to the several-mapping problem.

\section{A bounded monomial}\label{sec:frobenius}

The polynomial \eqref{eq:intro-polynomial} is a product of bases of linear
forms.  The only issue is to find one monomial whose exponents remain below
the field size.  Frobenius powers make this property invariant under every
change of coordinates.

Fix $n\ge1$, and let $K$ be a field of characteristic $p$.  For $s\ge1$,
let
\begin{equation}\label{eq:Js}
 J_s=(X_1^{p^s},\ldots,X_n^{p^s})
 \subseteq K[X_1,\ldots,X_n].
\end{equation}
This is the $p^s$-th Frobenius power of the coordinate ideal.
These ideals form the descending chain $J_1\supseteq J_2\supseteq\cdots$.
Since $J_s$ is a monomial ideal,
\begin{equation}\label{eq:bounded-monomial}
 f\notin J_s
 \quad\Longleftrightarrow\quad
 f\text{ has a nonzero monomial }X^\alpha
 \text{ with }\alpha_i<p^s\text{ for every }i.
\end{equation}

\begin{lemma}[Frobenius stability]\label{lem:stable}
If $T$ is any $n\times n$ matrix over $K$, the substitution
$\Phi_T(f)(X)=f(TX)$ satisfies $\Phi_T(J_s)\subseteq J_s$.  If $T$ is
invertible, then $\Phi_T(J_s)=J_s$.
\end{lemma}

\begin{proof}
Write the $i$th coordinate of $TX$ as
$\ell_i(X)=\sum_jt_{ij}X_j$.  Iterated Frobenius gives
\[
 \Phi_T(X_i^{p^s})=\ell_i(X)^{p^s}
 =\sum_jt_{ij}^{p^s}X_j^{p^s}\in J_s.
\]
Thus every generator of $J_s$ is carried into $J_s$, which proves the
inclusion.  If $T$ is invertible, applying the same inclusion to $T^{-1}$
gives the reverse inclusion.  This proves the lemma.
\end{proof}

For a basis $\mathcal B=\{\ell_1,\ldots,\ell_n\}$ of $(K^n)^*$, write
$P_{\mathcal B}=\ell_1\cdots\ell_n$.

\begin{proposition}[Products of bases]\label{prop:bases}
For $s\ge1$ and bases $\mathcal B_1,\ldots,\mathcal B_s$ of $(K^n)^*$,
\[
  \prod_{i=1}^sP_{\mathcal B_i}\notin J_s.
\]
Equivalently, the product has a nonzero monomial all of whose exponents are
smaller than $p^s$.
\end{proposition}

\begin{proof}
For $s=1$, change coordinates so that $\mathcal B_1$ is the coordinate
basis.  Its product becomes $X_1\cdots X_n\notin J_1$, and
Lemma~\ref{lem:stable} then gives the assertion in the original coordinates.

Suppose now that $s\ge2$ and that the assertion holds for $s-1$ bases.
Choose coordinates in which $\mathcal B_1$ is the coordinate basis.  The
images of $\mathcal B_2,\ldots,\mathcal B_s$ remain bases.  By induction,
their product has a monomial $X^\beta$ with $\beta_j<p^{s-1}$ for every
$j$.  Multiplication by $X_1\cdots X_n$ shifts this monomial to
$X^{\beta+(1,\ldots,1)}$ without changing its nonzero coefficient, and
\[
  \beta_j+1\le p^{s-1}<p^s.
\]
Thus the normalized product lies outside $J_s$.  Lemma~\ref{lem:stable}
gives the conclusion in the original coordinates.  This completes the
induction and proves the proposition.
\end{proof}

\begin{proof}[Proof of Theorem~\ref{thm:simultaneous}]
For $1\le i\le r-1$, let $\mathcal B_i$ be the basis of linear forms given
by the coordinates of $A_iX$.  Proposition~\ref{prop:bases} shows that
\[
 G(X):=\prod_{i=1}^{r-1}P_{\mathcal B_i}
\]
contains a monomial $cX^\beta$, with $c\ne0$ and
$\beta_j<p^{r-1}$ for every $j$.  Since
\[
 P(X)=(X_1\cdots X_n)G(X),
\]
the polynomial in \eqref{eq:intro-polynomial} contains
$cX^{\beta+(1,\ldots,1)}$.  Each of its exponents satisfies
\[
 \beta_j+1\le p^{r-1}<p^r\le p^k=q.
\]
This monomial has degree $rn=\deg P$.  The coefficient form of the
Combinatorial Nullstellensatz \cite[Theorem~1.2]{Alon1999}, applied with
every coordinate set equal to $\F_q$, gives an $x\in\F_q^n$ with
$P(x)\ne0$.  By the definition of $P$, the vectors
$x,A_1x,\ldots,A_{r-1}x$ are all nowhere zero.
\end{proof}

For $s=1$, Proposition~\ref{prop:bases} is the ideal-theoretic form of the
bounded-exponent step of Alon and Tarsi \cite[Claim~3]{AT1989}.

The prescribed-value theorem requires independent families that need not be
bases or have equal size.  The following corollary supplies this form.

\begin{corollary}[Independent-family form]\label{cor:independent-families}
Let $s\ge1$, and let $\mathcal L$ be a multiset of nonzero linear forms.  If
$\mathcal L$ can be covered by at most $s$ linearly independent subfamilies,
then
\[
  \prod_{\ell\in\mathcal L}\ell\notin J_s.
\]
\end{corollary}

\begin{proof}
If $\mathcal L$ is empty, its product is $1$, and the assertion is immediate.
Otherwise, assign each occurrence of a form to one independent subfamily
containing it.  Thus
$\mathcal L$ is the disjoint multiset union of independent families
$\mathcal L_1,\ldots,\mathcal L_t$, where $1\le t\le s$.

Extend each $\mathcal L_i$ to a basis, and let $g$ be the product of all
added forms.  Proposition~\ref{prop:bases} gives
\[
 g\prod_{\ell\in\mathcal L}\ell\notin J_t.
\]
If $\prod_{\ell\in\mathcal L}\ell$ belonged to $J_t$, multiplying it by
$g$ would place the displayed product in $J_t$, a contradiction.  Hence the
original product does not belong to $J_t$.  Since $J_s\subseteq J_t$, it
does not belong to $J_s$.  This proves the corollary.
\end{proof}

\section{Rectangular mappings and prescribed values}
\label{sec:values}

The argument does not require square matrices.  It also permits a different
forbidden set in every output coordinate.  For example, if $s\le k$, one
may prescribe one forbidden value in each coordinate of each of $s$
matrices with linearly independent rows.  The following form allows the
number of forbidden values to vary from one matrix to another.

\begin{theorem}[Simultaneous value avoidance]
\label{thm:matrix-lists}
Let $q=p^k$ and $s\ge1$.  For $1\le i\le s$, let
$B_i\in\F_q^{m_i\times n}$ have linearly independent rows, and let
$t_i$ be a nonnegative integer.  For each $i,j$, choose a set
$T_{ij}\subseteq\F_q$ with
$|T_{ij}|\le t_i$.  If
\[
 t_1+\cdots+t_s\le k,
\]
then there is an $x\in\F_q^n$ such that
\[
 (B_ix)_j\notin T_{ij}
 \qquad(1\le i\le s,\ 1\le j\le m_i).
\]
\end{theorem}

\begin{proof}
Let $\ell_{ij}(X)=(B_iX)_j$ and put
\[
 F(X)=\prod_{i=1}^s\prod_{j=1}^{m_i}
       \prod_{c\in T_{ij}}\bigl(\ell_{ij}(X)-c\bigr).
\]
If every forbidden set is empty, then $F=1$ and there is nothing to prove.
Otherwise let $t=t_1+\cdots+t_s$.  The top-degree homogeneous part of $F$ is
\[
 H(X)=\prod_{i=1}^s\prod_{j=1}^{m_i}
       \ell_{ij}(X)^{|T_{ij}|}.
\]
For fixed $i$ and $1\le h\le t_i$, take the forms $\ell_{ij}$ for which
$|T_{ij}|\ge h$.  Each resulting family is a subset of the rows of $B_i$
and hence is independent.  Counting multiplicity, these $t$ families contain
every linear factor of $H$.  Corollary~\ref{cor:independent-families} gives
$H\notin J_t$.  Since $t\le k$ and $J_k\subseteq J_t$, we also have
$H\notin J_k$.  Hence $H$ contains a monomial $cX^\alpha$ with $c\ne0$ and
$\alpha_j<p^k=q$ for every $j$.

Because $H$ is the top-degree homogeneous part of $F$, the same monomial
occurs in $F$, with the same coefficient and with degree $\deg F$.  The
Combinatorial Nullstellensatz, applied to the grid $\F_q^n$, gives a point
$x$ with $F(x)\ne0$.  Every factor in the displayed product is then nonzero,
which is the desired conclusion.
\end{proof}

Taking $B_1=I_n$, taking $B_2,\ldots,B_r$ to be
$A_1,\ldots,A_{r-1}$, and setting every forbidden set equal to $\{0\}$
shows that Theorem~\ref{thm:matrix-lists} contains
Theorem~\ref{thm:simultaneous}.  Both proofs use only linear independence
among the coordinate forms.  Matroids provide the coordinate-free language
for this information.

\section{A conjecture about matroids}\label{sec:matroids}

Let $M$ be a finite loopless matroid, and write $\si(M)$ for its
simplification.  Let $a(M)$ be the least number of independent sets needed
to cover $E(\si(M))$; thus $a(M)=0$ when $M$ is empty.  We pass to the
simplification because proportional vectors impose the same nowhere-zero
condition.  When $E(M)\ne\varnothing$, Edmonds's matroid-partition theorem
\cite{Edmonds} gives
\begin{equation}\label{eq:edmonds}
 a(M)=
 \max_{\varnothing\ne A\subseteq E(\si(M))}
 \left\lceil\frac{|A|}{\rk_{\si(M)}(A)}\right\rceil.
\end{equation}

If $M$ has rank $n$, its characteristic polynomial is
\[
 \chi_M(t)=\sum_{A\subseteq E(M)}
 (-1)^{|A|}t^{n-\rk_M(A)}.
\]

We now identify the number of avoiding functionals with the characteristic
polynomial.

\begin{lemma}[Finite-field method]\label{lem:finite-field}
Let nonzero vectors $(v_e)_{e\in E}$ span $\F_q^n$ and represent a matroid
$M$.  The number of functionals $\varphi\in(\F_q^n)^*$ such that
$\varphi(v_e)\ne0$ for every $e$ is $\chi_M(q)$.
\end{lemma}

\begin{proof}
For $A\subseteq E$, the functionals vanishing on every $v_e$ with $e\in A$
form the annihilator of $\operatorname{span}\{v_e:e\in A\}$.  There are
therefore $q^{n-\rk_M(A)}$ such functionals.  Inclusion--exclusion gives
\[
 \sum_{A\subseteq E}(-1)^{|A|}q^{n-\rk_M(A)}=\chi_M(q),
\]
the subset expansion of the characteristic polynomial.  This is the
one-functional case of the Crapo--Rota Critical Theorem.  It is also known
as the finite-field method \cite[Theorem~7.6.1]{Zaslavsky}; compare
\cite{Athanasiadis}.  This proves the lemma.
\end{proof}

The matrix problem suggests a general question: how large must the field be,
relative to $a(M)$, before $\chi_M(q)$ is forced to be positive?  An
unpublished conjecture of M.~J.~Moghaddamzadeh proposes the following answer
over prime fields.

\begin{conjecture}[Moghaddamzadeh, prime-field form]
\label{conj:prime-field}
Let $p$ be prime, and let $M$ be loopless and representable over $\F_p$.
If
\[
 p>2a(M)-1,
\]
then $\chi_M(p)>0$.
\end{conjecture}

Two hypotheses are essential: $M$ must be representable over the field at
which $\chi_M$ is evaluated, and that field must be prime.  For the first,
take $M=\PG(5,2)$.  Edmonds's formula and
the sizes of the projective subspaces give
\[
 a(M)=\left\lceil\frac{63}{6}\right\rceil=11.
\]
Yet the standard formula \cite{Oxley} gives
\[
 \chi_M(23)=\prod_{i=0}^{5}(23-2^i)<0,
 \qquad 23>2a(M)-1=21.
\]
Since $M$ has a Fano restriction, it is not representable over any field of
odd characteristic, in particular not over $\F_{23}$.

For the second, extend the standard $\F_2$-representation of
$M=\PG(3,2)$ to $\F_8$.  Then $M$ is simple and $\F_8$-representable,
\[
 a(M)=\left\lceil\frac{15}{4}\right\rceil=4,
 \qquad 8>2a(M)-1,
\]
but
\[
 \chi_M(8)=(8-1)(8-2)(8-4)(8-8)=0.
\]
This disproves the field-size assertion with $q>2a(M)-1$.  It does not
disprove the different characteristic assertion with $p>2a(M)-1$ over
$\F_{p^k}$; that problem remains open in general.

The examples above concern field size.  The simultaneous-mapping argument
instead gives a sufficient condition on the extension degree.

\begin{theorem}[Matroidal extension-degree theorem]
\label{thm:matroid-extension}
Let $M$ be a finite loopless matroid representable over $\F_{p^k}$.  If
\[
 a(M)\le k,
\]
then $\chi_M(p^k)>0$.
\end{theorem}

\begin{proof}
If $E(M)=\varnothing$, then $\chi_M(p^k)=1$, so assume that
$E(M)\ne\varnothing$.
Choose a spanning representation of $M$ in $\F_{p^k}^{\rk(M)}$, and choose
one vector from each of its projective classes.
These vectors represent $\si(M)$ and can be partitioned into $a(M)$
independent families.  Regard them as linear forms on the dual space.  By
Corollary~\ref{cor:independent-families}, their product lies outside
$J_{a(M)}$.  Since $a(M)\le k$, we have $J_k\subseteq J_{a(M)}$; hence the
product lies outside $J_k$ and contains a nonzero monomial all of whose
exponents are smaller than $p^k$.

The product is homogeneous, so this monomial has its full degree.  The
Combinatorial Nullstellensatz now gives a functional over $\F_{p^k}$
that is nonzero on every chosen vector.  Parallel copies impose the same
condition, so Lemma~\ref{lem:finite-field} gives $\chi_M(p^k)>0$.
\end{proof}

Conversely, suppose a represented point set is covered by $s\ge1$ independent
sets.  Assign each point to one covering set, thereby partitioning the point
set into independent parts.  Complete each part to a basis and take the first
basis as coordinates.  The remaining bases are the row sets of $s-1$
invertible matrices.  Thus
Theorem~\ref{thm:matroid-extension} is precisely the coordinate-free form of the
simultaneous-mapping theorem.

When a representation descends to a subfield $\F_{p^j}\subseteq\F_{p^k}$,
another unpublished result of M.~J.~Moghaddamzadeh, proved by a direct
construction with scalar coordinates, gives positivity under the stronger
hypothesis $a(M)\le k/j$.
Theorem~\ref{thm:matroid-extension} needs no descent assumption.  We next study the
prime-field conjecture through the geometry of its obstructions.

\section{Blocking-set density in ranks two and three}
\label{sec:blocking}

Fix a prime power $q$.  The equality $\chi_M(q)=0$ has a geometric form: the
representing projective points form a spanning set that meets every
hyperplane.  The least density of such a set measures the failure of
$q$-avoidance.  The projective line and plane give exact obstruction
thresholds in ranks two and three.

\subsection{Density and the obstruction threshold}

Call $B\subseteq\PG(n-1,q)$ a \emph{spanning blocker} if it spans the
projective space and meets every projective hyperplane.  Let $M(B)$ be its
simple vector matroid, and define its integral density by
\begin{equation}\label{eq:density}
 \delta(B):=
 \max_{\substack{U\le\F_q^n\\B\cap\PG(U)\ne\varnothing}}
 \left\lceil\frac{|B\cap\PG(U)|}{\dim U}\right\rceil.
\end{equation}
Edmonds's formula gives $\delta(B)=a(M(B))$, while
Lemma~\ref{lem:finite-field} gives
\begin{equation}\label{eq:blocking-equivalence}
 \chi_{M(B)}(q)=0
 \quad\Longleftrightarrow\quad
 B\text{ meets every projective hyperplane}.
\end{equation}
Deleting parallel copies changes neither the avoidance condition nor the
characteristic polynomial.  Thus every $\F_q$-representable loopless
obstruction may be represented by a spanning blocker $B$, with
$a(M(B))=\delta(B)$.

For $n\ge2$, let $\rmax(q,n)$ be the largest integer $r$ such that every
loopless rank-$n$ matroid $M$ representable over $\F_q$ and satisfying
$a(M)\le r$ has $\chi_M(q)>0$.  Put
\[
 \rmax(q):=\inf_{n\ge2}\rmax(q,n).
\]
Equivalently,
\begin{equation}\label{eq:rmax-blocking}
 \rmax(q,n)+1
 =\min\{\delta(B):B\subseteq\PG(n-1,q)
       \text{ is a spanning blocker}\}.
\end{equation}

\begin{lemma}\label{lem:rmax-monotone}
For fixed $q$, the sequence $\rmax(q,n)$ is nonincreasing in $n$.
\end{lemma}

\begin{proof}
Embed a minimizing blocker $B$ in a hyperplane $H$ of $\PG(n,q)$ and adjoin
a point $P$ outside $H$.  Every hyperplane other than $H$ cuts $H$ in a
hyperplane met by $B$, so $B\cup\{P\}$ is a spanning blocker in $\PG(n,q)$.
Its density is at least $\delta(B)$ because it contains $B$ in $H$.  For the
reverse inequality, let $U$ be a subspace.  If $U$ contains $P$, then it
contains at most
$\delta(B)(\dim U-1)+1\le\delta(B)\dim U$ points of $B\cup\{P\}$.  If
$P\notin U$, all points of $B\cup\{P\}$ that lie in $U$ belong to
$U\cap H$, whose dimension is at most $\dim U$.  Equation
\eqref{eq:rmax-blocking} proves the lemma.
\end{proof}

\Needspace{20\baselineskip}
\subsection{Ranks two and three}

Let $\tau(q)$ be the least size of a blocking set in $\PG(2,q)$ that
contains no line.  Put $\tau(2)=+\infty$, since the Fano plane has no such
set.  This classical planar number gives the rank-three threshold exactly.

\begin{theorem}[Ranks two and three]\label{thm:rank23}
For every prime power $q$,
\begin{align}
 \rmax(q,2)&=\left\lfloor\frac q2\right\rfloor,\label{eq:rank2}\\
 \rmax(q,3)&=
 \min\left\{\left\lfloor\frac q2\right\rfloor,\,
   \left\lceil\frac{\tau(q)}3\right\rceil-1\right\}.
 \label{eq:rank3exact}
\end{align}
If $p$ is prime, then
\begin{equation}\label{eq:rank3prime}
 \rmax(p,3)=\left\lfloor\frac p2\right\rfloor.
\end{equation}
If $q$ is a square, then
\begin{equation}\label{eq:rank3square}
 \rmax(q,3)=
 \min\left\{\left\lfloor\frac q2\right\rfloor,\,
   \left\lceil\frac{q+\sqrt q+1}{3}\right\rceil-1\right\}.
\end{equation}
\end{theorem}

\begin{proof}
In $\PG(1,q)$ the hyperplanes are points, so the unique blocking set is the
whole line of $q+1$ points.  Its density is
$\lceil(q+1)/2\rceil$, proving \eqref{eq:rank2}.

Now let $B$ be a spanning blocking set in $\PG(2,q)$.  If $B$ contains a
line, then
\[
 \delta(B)\ge\left\lceil\frac{q+1}{2}\right\rceil.
\]
If it contains no line, then $|B|\ge\tau(q)$, and hence
\[
 \delta(B)\ge\left\lceil\frac{\tau(q)}3\right\rceil.
\]
This gives the lower bound for the minimum blocker density.

A line together with one external point is a spanning blocker of density
$\lceil(q+1)/2\rceil$.  For $q\ge3$, put $L_0:Z=0$ and
$P_0=(1:0:0)$.  The set
\[
 (L_0\setminus\{P_0\})\cup\{(1:0:1)\}
 \cup\{(0:t:1):t\in\F_q^\times\}.
\]
meets every line through $P_0$, while every other line meets $L_0$ away
from $P_0$.  It contains no line.  Indeed, any contained line other than
$L_0$ would contain all $q$ displayed affine points; but two points
$(0:t:1)$ determine the line $X=0$, which does not contain $(1:0:1)$.
Thus $\tau(q)\le2q$.

Now choose a line-free blocker $B_0$ of size $\tau(q)$.  For any line
$\ell$, choose $R\in\ell\setminus B_0$.  The other $q$ lines through $R$
require distinct points of $B_0\setminus\ell$, so
$|B_0\cap\ell|\le\tau(q)-q$.  Since $\tau(q)\le2q$, we have
\[
 \left\lceil\frac{|B_0\cap\ell|}{2}\right\rceil
 \le
 \left\lceil\frac{\tau(q)}3\right\rceil.
\]
The term in $\delta(B_0)$ corresponding to the whole plane is the right-hand
side.  Hence
$\delta(B_0)=\lceil\tau(q)/3\rceil$.  Together with the preceding two
lower bounds, these examples prove \eqref{eq:rank3exact}; the case $q=2$
comes from the line-plus-point example.

For an odd prime $p$, Blokhuis's theorem and the projective-triangle
construction give $\tau(p)=3(p+1)/2$ \cite{Blokhuis}; $p=2$ was handled
above.  This yields \eqref{eq:rank3prime}.  For square $q$, Bruen's theorem
and a Baer subplane give $\tau(q)=q+\sqrt q+1$ \cite{Bruen}, which yields
\eqref{eq:rank3square}.  This completes the proof.
\end{proof}

The exact formulas have an immediate consequence for
Conjecture~\ref{conj:prime-field}.  In fact they prove the stronger
characteristic statement through rank three.

\begin{corollary}[Characteristic threshold in low rank]
\label{cor:rank3conjecture}
Let $M$ be a loopless matroid of rank at most three, represented over
$\F_q$, where $q$ has characteristic $p$.  If
$p>2a(M)-1$, then $\chi_M(q)>0$.
\end{corollary}

\begin{proof}
Put $r=a(M)$.  Ranks zero and one are immediate: their simplified
characteristic polynomials are respectively $1$ and $t-1$.  The hypothesis
implies $r\le\lfloor p/2\rfloor$.  In rank two,
\eqref{eq:rank2} gives
$\rmax(q,2)=\lfloor q/2\rfloor\ge\lfloor p/2\rfloor$.  In rank three over
$q=p$, use \eqref{eq:rank3prime}.  If $q\ge p^2$, then
\[
 \left\lfloor\frac q2\right\rfloor\ge\left\lfloor\frac p2\right\rfloor,
 \qquad
 \left\lceil\frac{\tau(q)}{3}\right\rceil-1
 \ge
 \left\lceil\frac{q+\sqrt q+1}{3}\right\rceil-1
 \ge\left\lfloor\frac p2\right\rfloor.
\]
Bruen's theorem supplies the inequality
$\tau(q)\ge q+\sqrt q+1$.  Thus both
terms in \eqref{eq:rank3exact} are at least $r$.  This proves the
corollary.
\end{proof}

\section{From the plane to higher rank}\label{sec:projection}

The number $\tau(q)$ controls more than rank three.  Projection from a point
outside $B$ sends a spanning blocker to one rank lower.  If the center lies
on a secant of $B$, two points have the same image, and hence
$|B|\ge|\pi_P(B)|+1$.  We first use this observation in rank four and then
iterate it.

We call a projective subspace \emph{rank $j$} when its underlying vector
space has dimension $j$.  Thus lines, planes, and solids have ranks two,
three, and four.

\begin{lemma}[Projection]\label{lem:projection}
Let $d\ge4$, let $B\subseteq\PG(d-1,q)$ be a spanning blocker, let
$P\notin B$, and let $H$ be a hyperplane not containing $P$.  Write
\[
 \pi_P(X)=\langle P,X\rangle\cap H,\qquad B'=\pi_P(B).
\]
Then:
\begin{enumerate}
\item $B'$ spans $H$ and meets every hyperplane of $H$;
\item if a line $\ell\subseteq H$ lies in $B'$, then
      $|B\cap\langle P,\ell\rangle|\ge q+1$;
\item if $P$ lies on a line containing $t$ points of $B$, then
      $|B'|\le |B|-t+1$.
\end{enumerate}
\end{lemma}

\begin{proof}
If $B'$ lay in a proper subspace $K$ of $H$, then $B$ would lie in the
proper subspace $\langle P,K\rangle$, contradicting the assumption that
$B$ spans.  If $J$ is a
hyperplane of $H$, the ambient hyperplane $\langle P,J\rangle$ meets $B$,
and the image of a point of intersection lies in $J$.  This proves the first
assertion.

For the second, choose one preimage of each of the $q+1$ points of $\ell$.
These preimages are distinct and lie in $\langle P,\ell\rangle$.  Finally,
the $t$ collinear points in the third assertion have one image, so the
projection loses at least $t-1$ points.
\end{proof}

The following count will be used in both projection arguments: if a
hyperplane of a proper subspace misses $B$, then at least $q$ points of $B$
lie outside that subspace.

\begin{lemma}[Missing-subspace count]\label{lem:missing-flat}
Let $B\subseteq\PG(d-1,q)$ meet every hyperplane.  Let $S$ be a proper
rank-$s$ subspace, and let $T$ be a hyperplane of $S$.  If $T\cap B$ is
empty, then
\[
 |B\setminus S|\ge q.
\]
\end{lemma}

\begin{proof}
There are $q^{d-s}$ ambient hyperplanes through $T$ that do not contain
$S$.  Each contains a point of $B\setminus S$.  A fixed point outside $S$
lies in $q^{d-s-1}$ of these hyperplanes.  Counting incident pairs gives
$q^{d-s}\le |B\setminus S|q^{d-s-1}$, as required.
\end{proof}

\subsection{The rank-four bound}

In rank four, one projection reaches the projective plane, where the exact
parameter $\tau(q)$ is available.

\begin{theorem}[Rank-four projection bound]\label{thm:rank4projection}
Every spanning blocker $B\subseteq\PG(3,q)$ satisfies at least one of the
following alternatives:
\begin{enumerate}
\item $B$ contains a line;
\item for some plane $\Pi$, the set $B\cap\Pi$ is a line-free planar
      blocker;
\item $|B|\ge2q+1$;
\item $|B|\ge\tau(q)+1$.
\end{enumerate}
In particular,
\begin{equation}\label{eq:rank4-bound}
 \delta(B)\ge
 \min\left\{
   \left\lceil\frac{2q+1}{4}\right\rceil,\,
   \left\lceil\frac{\tau(q)+1}{4}\right\rceil
 \right\}.
\end{equation}
Consequently
\[
 \rmax(q,4)\ge
 \min\left\{
   \left\lceil\frac{2q+1}{4}\right\rceil,\,
   \left\lceil\frac{\tau(q)+1}{4}\right\rceil
 \right\}-1.
\]
\end{theorem}

\begin{proof}
If $B$ contains a line, the first alternative holds.  Otherwise choose
distinct $a,b\in B$ and a point $P\in ab\setminus B$, and project onto a
plane $H$ not containing $P$.  By Lemma~\ref{lem:projection}, the image
$B'$ is a planar blocking set and $\pi_P(a)=\pi_P(b)$.

Suppose first that $B'$ contains a line $\ell$.  The plane
$\Pi=\langle P,\ell\rangle$ contains at least $q+1$ points of $B$.  If a
line of $\Pi$ misses $B$, Lemma~\ref{lem:missing-flat} gives
$|B\setminus\Pi|\ge q$, and hence $|B|\ge2q+1$, the third alternative.
Otherwise $B\cap\Pi$ is a line-free planar blocker, the second alternative.

If $B'$ contains no line, then $|B'|\ge\tau(q)$.  The collapsed pair
$a,b$ gives $|B|\ge|B'|+1\ge\tau(q)+1$, the fourth alternative.

The first and third alternatives give
$\delta(B)\ge\lceil(2q+1)/4\rceil$.  The second gives
$\delta(B)\ge\lceil\tau(q)/3\rceil
\ge\lceil(\tau(q)+1)/4\rceil$, and the fourth gives the latter bound
directly.  This proves \eqref{eq:rank4-bound} and its consequence for
$\rmax(q,4)$.
\end{proof}

\begin{corollary}[Prime fields and proper extension fields]
\label{cor:rank4-consequences}
The rank-four bound has the following consequences.
\begin{enumerate}
\item For every odd prime $p$,
\[
 \rmax(p,4)\ge\left\lceil\frac{3p+5}{8}\right\rceil-1.
\]
\item For $p\in\{2,3,5,7\}$,
\[
 \rmax(p,4)=\left\lfloor\frac p2\right\rfloor.
\]
\item Let $q=p^k$ with $k\ge2$.  If a loopless $\F_q$-representable
matroid $M$ has rank at most four and
$p>2a(M)-1$, then $\chi_M(q)>0$.
\end{enumerate}
\end{corollary}

\begin{proof}
For odd $p$, Blokhuis's theorem and the projective-triangle construction give
$\tau(p)=3(p+1)/2$.  The second term in \eqref{eq:rank4-bound} becomes
$\lceil(3p+5)/8\rceil$, which is no larger than the first.  This proves the
first assertion.  Lemma~\ref{lem:rmax-monotone} and
Theorem~\ref{thm:rank23} give
$\rmax(p,4)\le\lfloor p/2\rfloor$; the first assertion gives equality for
$p=3,5,7$, while Theorem~\ref{thm:matroid-extension} handles $p=2$.

For the last assertion, put $r=a(M)$.  The hypothesis gives
$r\le\lfloor p/2\rfloor$.  Ranks at most three follow from
Corollary~\ref{cor:rank3conjecture}.  In rank four, Bruen's theorem gives
\[
 \tau(q)\ge q+\sqrt q+1\ge p^2+p+1.
\]
Since $q\ge p^2$, both terms on the right of \eqref{eq:rank4-bound} are at
least $\lfloor p/2\rfloor+1$.  Thus $r\le\rmax(q,4)$, which proves the
claim.
\end{proof}

\subsection{Iteration in higher rank}

Repeated projection either produces a dense lower-rank section or decreases
the size of the image by at least one at each step.  The three terms below correspond, respectively, to a
section with at least $2q+1$ points, a lower bound for $|B|$, and a proper
section with at least $\tau(q)+1$ points.

\begin{theorem}[Projection in every rank]\label{thm:projection-all-ranks}
Let $d\ge4$.  Every spanning blocker $B\subseteq\PG(d-1,q)$ satisfies
\begin{equation}\label{eq:all-rank-projection}
 \delta(B)\ge
 \min\left\{
  \left\lceil\frac{2q+1}{d}\right\rceil,\,
  \left\lceil\frac{\tau(q)+d-3}{d}\right\rceil,\,
  \left\lceil\frac{\tau(q)+1}{d-1}\right\rceil
 \right\}.
\end{equation}
Consequently, $\rmax(q,d)$ is at least the right-hand side minus one.
\end{theorem}

\begin{proof}
If $q\le d-4$, the first term on the right is at most two.  A spanning
blocker cannot be a projective basis: in basis coordinates, the hyperplane
$X_1+\cdots+X_d=0$ misses every basis point.  Thus $|B|\ge d+1$ and
$\delta(B)\ge2$, proving this case.  We may therefore assume
$q\ge d-3$.

We prove by induction on $d$ that at least one of the following alternatives
holds:
\begin{enumerate}
\item[(a)] $B$ contains a line;
\item[(b)] for some plane $\Pi$, the set $B\cap\Pi$ contains a line-free
      planar blocker $C$;
\item[(c)] for some rank-$j$ subspace $U$, with $4\le j\le d$,
      $|B\cap U|\ge2q+1$;
\item[(d)] for some rank-$j$ subspace $U$, with $4\le j\le d-1$,
      $|B\cap U|\ge\tau(q)+1$;
\item[(e)] $|B|\ge\tau(q)+d-3$.
\end{enumerate}
Each alternative implies \eqref{eq:all-rank-projection}.  Alternative (a)
gives $\delta(B)\ge\lceil(q+1)/2\rceil
\ge\lceil(2q+1)/d\rceil$, and (c) also gives the first displayed term.
Alternative (d) gives the third term, while (e) gives the second.
Finally, (b) gives density at least $\lceil\tau(q)/3\rceil$; this is at least
$\lceil(\tau(q)+1)/4\rceil$ when $d=4$, and at least the third term when
$d\ge5$.  When $q=2$, alternatives (b), (d), and (e) cannot occur.

For $d=4$, Theorem~\ref{thm:rank4projection} gives (a), (b), (c), or (e).
Suppose now that $d\ge5$ and that $B$ contains no line.
Choose $P\notin B$ on a secant and project onto a hyperplane $H$.  By
Lemma~\ref{lem:projection}, the image $B'$ spans $H$, blocks its
hyperplanes, and satisfies $|B'|\le|B|-1$.  Apply the induction hypothesis
to $B'$.

If $B'$ contains a line, its inverse image lies in a plane $\Pi$ containing
at least $q+1$ points of $B$.  A line of $\Pi$ missing $B$ gives
$|B|\ge2q+1$ by Lemma~\ref{lem:missing-flat}; if no line is missed, then
$B\cap\Pi$ is a line-free planar blocker.  Thus (c) or (b) holds.

Next suppose that a plane $\Pi'\subseteq H$ contains a line-free blocker
$C\subseteq B'$.  Put $S=\langle P,\Pi'\rangle$, a solid.  Then
$|B\cap S|\ge|C|\ge\tau(q)$.  If a plane of $S$ misses $B$, the
missing-subspace count gives
\[
 |B|\ge\tau(q)+q\ge\tau(q)+d-3,
\]
which is (e).

Otherwise $B\cap S$ meets every plane of $S$.  Put
$W=\langle B\cap S\rangle$.  Since $C$ spans $\Pi'$, there are only two
possibilities: $W$ is a plane not containing $P$, or $W=S$.  Indeed, no
smaller subspace can project onto $\Pi'$, and a plane through $P$ projects
to a line.  In the first case, projection restricts to a collineation
$W\to\Pi'$, and the inverse image of $C$ gives (b).  In the second, apply
Theorem~\ref{thm:rank4projection} inside $S$.  Its four alternatives give,
respectively, (a), (b), (c), and (d); the last is available because
$d\ge5$.

It remains to lift alternatives (c)--(e).  A rank-$j$ subspace
$U'\subseteq H$ lifts to the rank-$(j+1)$ subspace
$U=\langle P,U'\rangle$, and
$|B\cap U|\ge|B'\cap U'|$.  Thus every instance of (c) or (d) below the
top rank lifts to the same alternative one rank higher.  A top-rank instance
of (c) gives $|B|\ge2q+1$.  Finally, if
$|B'|\ge\tau(q)+(d-1)-3$, then $|B'|\le|B|-1$ gives
\[
 |B|\ge|B'|+1\ge\tau(q)+d-3.
\]
This is (e), and the induction is complete.
\end{proof}

Projection supplies lower bounds in every fixed rank.  We now turn in the
opposite direction: embedded subgeometries produce obstructions in
unbounded rank and lead back to the original problem of simultaneous
mappings.

\section{Subgeometries and sharpness}
\label{sec:sharpness}

Projective subgeometries furnish blockers whose density can be computed
exactly.  They show that the extension-degree bound is sharp for four fields
and give lower bounds for the field size required without an extension-degree
hypothesis.

\subsection{The extension-degree bound}

The relevant obstruction is the subgeometry
$\PG(k,p)\subseteq\PG(k,p^k)$.  For a prime $p$ and an integer $k\ge1$, set
\begin{equation}\label{eq:dpk}
 d_{p,k}:=
 \left\lceil
 \frac{p^{k+1}-1}{(k+1)(p-1)}
 \right\rceil.
\end{equation}

\begin{proposition}[Subgeometry obstruction]\label{prop:subgeometry}
Let $q=p^k$.  The projective geometry $\PG(k,p)$, viewed over $\F_q$, is a
spanning blocker in $\PG(k,q)$ and has density $d_{p,k}$.
Consequently
\begin{equation}\label{eq:global-bounds}
 k\le\rmax(p^k)\le
 \min\left\{\left\lfloor\frac{p^k}{2}\right\rfloor,\,
             d_{p,k}-1\right\}.
\end{equation}
In particular,
\[
 \rmax(q)=k\qquad\text{for }q\in\{2,3,4,8\}.
\]
\end{proposition}

\begin{proof}
The standard coordinate points show that $\PG(k,p)$ spans over $\F_q$.
An $\F_q$-linear functional on $\F_q^{k+1}$ restricts to an
$\F_p$-linear map from a space of dimension $k+1$ to $\F_q$, which has
$\F_p$-dimension $k$.  Its kernel therefore contains a nonzero
$\F_p$-vector.  Hence every hyperplane meets the subgeometry.

If $U\le\F_q^{k+1}$ has dimension $j$, then
$U\cap\F_p^{k+1}$ has $\F_p$-dimension at most $j$, since vectors with
$\F_p$ coordinates that are independent over $\F_p$ remain independent
after extension of scalars to $\F_q$.  Hence
$\PG(k,p)\cap\PG(U)$ has at most $(p^j-1)/(p-1)$ points, and
\[
 \frac{p^j-1}{j(p-1)}
\]
is increasing in $j$.  Indeed, after positive denominators are cleared, the
inequality between consecutive terms reduces to
\[
 p^j\bigl(j(p-1)-1\bigr)+1\ge0.
\]
Formula \eqref{eq:density} thus gives density
$d_{p,k}$.  The lower bound in \eqref{eq:global-bounds} is
Theorem~\ref{thm:matroid-extension}; the two upper bounds follow from this obstruction
and \eqref{eq:rank2}.  Substitution gives equality for the four listed
fields: for $q\in\{2,3,4,8\}$, the corresponding upper bound equals $k$.
This proves the proposition.
\end{proof}

For these four fields, Theorem~\ref{thm:matroid-extension} therefore admits
no uniform-in-rank improvement.  The same subgeometry also violates the
field-size condition $q>2a(M)-1$ in the proper extensions specified next.

\Needspace{12\baselineskip}
\begin{corollary}[Simple counterexamples to a field-size threshold]
\label{cor:simple-counterexamples}
Suppose that either $k\ge3$, or $k=2$ and $p\ge5$.  For every integer
\[
 d_{p,k}\le r\le\left\lfloor\frac{p^k}{2}\right\rfloor
\]
there is a simple $\F_{p^k}$-representable matroid $M_{p,k,r}$ such that
\[
 a(M_{p,k,r})=r,\qquad
 p^k>2r-1,\qquad
 \chi_{M_{p,k,r}}(p^k)=0.
\]
\end{corollary}

\begin{proof}
The stated interval for $r$ is nonempty because the inequality
$d_{p,k}\le\lfloor p^k/2\rfloor$ is equivalent to
\[
 1+p+\cdots+p^k\le(k+1)\left\lfloor\frac{p^k}{2}\right\rfloor.
\]
For $p=2$ and $k\ge3$, this follows from
$2^{k+1}-1\le(k+1)2^{k-1}$.  For odd $p$ and $k\ge3$, the left-hand side is
at most $3p^k/2\le2(p^k-1)$, which is no larger than the right-hand side.
When $k=2$ and $p\ge5$, the displayed inequality is equivalent to
$p^2-2p-5\ge0$.  Now take
\[
 M_{p,k,r}=\PG(k,p)\oplus U_{2,2r}.
\]
Here $U_{2,2r}$ is the rank-two uniform matroid on $2r$ elements.  It is
simple and representable over $\F_{p^k}$ because $2r\le p^k$.  The covering
number is the maximum under direct sum, so it equals $r$.  Characteristic
polynomials multiply, and
$\chi_{\PG(k,p)}(p^k)=0$.  Finally,
$r\le\lfloor p^k/2\rfloor$ implies $p^k>2r-1$.
This proves the corollary.
\end{proof}

The summand $U_{2,2r}$ raises the covering number to $r$ while preserving
simplicity.

\subsection{Returning to several mappings}

The matrix--matroid correspondence following
Theorem~\ref{thm:matroid-extension} turns a blocker covered by $r$
independent sets into $r$ invertible matrices for which no vector has all
$r$ images nowhere zero.  Blocking sets therefore give lower bounds for the
large-field threshold of Nagy, Pach, and Tomon.

For an integer $r\ge2$, let $Q(r)$ be the infimum of the real numbers $Q$
such that, for every prime power $q>Q$, every $n\ge1$, and every family
$A_1,\ldots,A_r\in\operatorname{GL}_n(\F_q)$, there is an $x\in\F_q^n$ for
which every $A_ix$ is nowhere zero.  Nagy, Pach, and Tomon proved that $Q(r)$
is finite \cite[Theorem~1.7]{NagyPachTomon}.  The preceding obstructions give
the following quantitative lower bounds.

\begin{proposition}[Lower bounds for the large-field threshold]
\label{prop:large-field-lower}
Let $q_r$ denote the largest prime power not exceeding $2r-1$.  Then
\[
 Q(r)\ge q_r.
\]
Moreover, for every integer $m\ge1$, put
\[
 d_m:=d_{2,m}=\left\lceil\frac{2^{m+1}-1}{m+1}\right\rceil.
\]
If $r\ge d_m$, then $Q(r)\ge2^m$.  Consequently, as $m\to\infty$,
\[
 Q(d_m)\ge\left(\frac12-o(1)\right)d_m\log_2d_m,
\]
whereas, for arbitrary $r\to\infty$,
\[
 Q(r)\ge\left(\frac14-o(1)\right)r\log_2r.
\]
\end{proposition}

\begin{proof}
For every prime power $q\le2r-1$, the full projective line is a blocker of
density $\lceil(q+1)/2\rceil\le r$.  Partitioning it into at most $r$
independent sets, adding empty classes if necessary, and completing every
class to a basis gives $r$ invertible matrices for which no vector has all
$r$ images nowhere zero.  Indeed, the coordinate forms of each matrix
constitute a basis, and the original blocking forms occur among their union.
This proves the first assertion.

For $q=2^m$, the embedded subgeometry $\PG(m,2)$ is a blocker of density
$d_m$, proving the second.  Since
\[
 d_m=(1+o(1))\frac{2^{m+1}}{m+1},
\]
we have $m=(1+o(1))\log_2d_m$ and hence
\[
 2^m=\left(\frac12+o(1)\right)d_m\log_2d_m.
\]
This proves the estimate along $r=d_m$.  For arbitrary $r$, choose the
largest $m$ with $d_m\le r$.  The obstruction for $d_m$ matrices may be
padded with repeated matrices, so it is also an obstruction for $r$
matrices.  Moreover,
$r<d_{m+1}=(2+o(1))d_m$ and $m=(1+o(1))\log_2r$.  Therefore
\[
 2^m\ge\left(\frac14-o(1)\right)r\log_2r,
\]
which proves the uniform estimate.
\end{proof}

\section{Open problems}\label{sec:questions}

Theorem~\ref{thm:simultaneous} solves the simultaneous-mapping problem when
the extension degree is at least the number of mappings.  At the other
extreme, Nagy, Pach, and Tomon show that, for each fixed number of mappings,
a sufficiently large field always works.  Proposition~\ref{prop:large-field-lower}
shows that this large-field threshold grows at least on the order of
$r\log r$.

\begin{question}\label{ques:large-field}
Determine the order of growth of $Q(r)$.
\end{question}

Equation~\eqref{eq:rmax-blocking} identifies $\rmax(q)+1$ with a minimum
blocker density.  For every fixed $r$, the theorem of Nagy, Pach, and Tomon
\cite[Theorem~1.7]{NagyPachTomon} implies $\rmax(q)\ge r$ for all sufficiently
large prime powers $q$: partition a cover into at most $r$ independent
families, complete them to bases, and apply their matrix theorem.  Hence,
for fixed $k$,
\[
 \rmax(p^k)\longrightarrow\infty\qquad(p\longrightarrow\infty),
\]
even though Proposition~\ref{prop:subgeometry} gives
$\rmax(q)=k$ for $q\in\{2,3,4,8\}$.

\begin{question}\label{ques:rmax}
Determine $\rmax(q)$, or its order of growth, in terms of the characteristic
and the extension degree.
\end{question}

The exact rank-three formula and
Corollary~\ref{cor:rank4-consequences} prove the
characteristic analogue of Moghaddamzadeh's conjecture in ranks at most four
over proper extension fields.  The general case remains open.

\begin{question}\label{ques:characteristic}
Suppose that $M$ is loopless and representable over $\F_{p^k}$, with
$p>2a(M)-1$.  Must $\chi_M(p^k)$ be positive?
\end{question}

Products of bases solve the matrix problem by producing a monomial with
bounded exponents.  A cover by independent sets is exactly the matroidal
hypothesis needed for this argument.  Spanning blockers are exactly the
represented configurations for which $\chi_M(q)=0$.  Thus the algebra proves
existence, and projective geometry describes failure.

\section*{Acknowledgments}

The author thanks M.~J.~Moghaddamzadeh for sharing his conjecture and for
helpful discussions.


\small
\begin{thebibliography}{99}

\bibitem{Alon1999}
N.~Alon,
\newblock Combinatorial Nullstellensatz,
\newblock \emph{Combinatorics, Probability and Computing} 8 (1999),
nos.~1--2, 7--29,
\newblock
\href{https://doi.org/10.1017/S0963548398003411}
{doi:10.1017/S0963548398003411}.

\bibitem{AT1989}
N.~Alon and M.~Tarsi,
\newblock A nowhere-zero point in linear mappings,
\newblock \emph{Combinatorica} 9 (1989), no.~4, 393--395,
\newblock
\href{https://doi.org/10.1007/BF02125351}
{doi:10.1007/BF02125351}.

\bibitem{Athanasiadis}
C.~A.~Athanasiadis,
\newblock Characteristic polynomials of subspace arrangements and finite
fields,
\newblock \emph{Advances in Mathematics} 122 (1996), no.~2, 193--233,
\newblock
\href{https://doi.org/10.1006/aima.1996.0059}
{doi:10.1006/aima.1996.0059}.

\bibitem{Blokhuis}
A.~Blokhuis,
\newblock On the size of a blocking set in $\PG(2,p)$,
\newblock \emph{Combinatorica} 14 (1994), no.~1, 111--114,
\newblock
\href{https://doi.org/10.1007/BF01305953}
{doi:10.1007/BF01305953}.

\bibitem{Bruen}
A.~A.~Bruen,
\newblock Blocking sets in finite projective planes,
\newblock \emph{SIAM Journal on Applied Mathematics} 21 (1971), no.~3,
380--392,
\newblock
\href{https://doi.org/10.1137/0121041}
{doi:10.1137/0121041}.

\bibitem{Edmonds}
J.~Edmonds,
\newblock Minimum partition of a matroid into independent subsets,
\newblock \emph{Journal of Research of the National Bureau of Standards
Section B} 69B (1965), 67--72.

\bibitem{NagyPach}
J.~Nagy and P.~P.~Pach,
\newblock The Alon--Jaeger--Tarsi conjecture via group ring identities,
\newblock \emph{Journal of the European Mathematical Society}, online
first (2025),
\newblock
\href{https://doi.org/10.4171/JEMS/1640}{doi:10.4171/JEMS/1640}.

\bibitem{NagyPachTomon}
J.~Nagy, P.~P.~Pach, and I.~Tomon,
\newblock Hyperplane covers of finite spaces and applications,
\newblock \emph{Transactions of the American Mathematical Society} 379
(2026), no.~1, 137--156,
\newblock
\href{https://doi.org/10.1090/tran/9483}{doi:10.1090/tran/9483}.

\bibitem{Oxley}
J.~Oxley,
\newblock \emph{Matroid Theory},
\newblock second edition, Oxford University Press, Oxford, 2011.

\bibitem{Zaslavsky}
T.~Zaslavsky,
\newblock The M\"obius function and the characteristic polynomial,
\newblock in N.~White (ed.), \emph{Combinatorial Geometries},
Encyclopedia of Mathematics and its Applications, vol.~29,
Cambridge University Press, Cambridge, 1987, pp.~114--138.

\end{thebibliography}
\end{document}